\documentclass[12pt]{amsart}
\usepackage{SSdefn}
\usepackage{comment}

\newcommand{\gen}{{\rm gen}}

\newcommand{\fgen}{{\rm fg}}

\newcommand{\FI}{\mathbf{FI}}

\DeclareMathOperator{\sh}{sh}

\author{Steven V Sam}
\address{Department of Mathematics, University of Wisconsin, Madison, WI}
\email{\href{mailto:svs@math.wisc.edu}{svs@math.wisc.edu}}
\urladdr{\url{http://math.wisc.edu/~svs/}}
\thanks{SS was supported by a Miller research fellowship and NSF grant DMS-1500069.}

\author{Andrew Snowden}
\address{Department of Mathematics, University of Michigan, Ann Arbor, MI}
\email{\href{mailto:asnowden@umich.edu}{asnowden@umich.edu}}
\urladdr{\url{http://www-personal.umich.edu/~asnowden/}}
\thanks{AS was supported by NSF grants DMS-1303082 and DMS-1453893 and a Sloan fellowship.}

\title{Regularity bounds for twisted commutative algebras}


\date{April 5, 2017}

\begin{document}

\begin{abstract}
Let $A$ be the twisted commutative algebra freely generated by $d$ indeterminates of degree~1. We show that the regularity of an $A$-module can be bounded from the first $\lfloor \tfrac{1}{4} d^2 \rfloor+2$ terms of its minimal free resolution. This extends results of Church and Ellenberg from the $d=1$ case.
\end{abstract}

\maketitle

\section{Introduction}

\subsection{Statement of results}

Let $E$ be a $d$-dimensional complex vector space, and let $A=\Sym(E \otimes \bC^{\infty})$, a polynomial ring in infinitely many variables equipped with a natural action of $\GL_{\infty}$. We treat $A$ as a twisted commutative algebra (tca); in fact, it is the free tca on $d$ indeterminates of degree~1. Modules over $A$ are required to have a compatible action of $\GL_{\infty}$ satisfying the technical condition of polynomiality (see \S \ref{s:bg}).

Let $M$ be an $A$-module, and consider the minimal free resolution $F_{\bullet} \to M$. This is almost always infinite: in fact, it is infinite whenever $M$ is not projective. However, we showed in \cite{symu1} that, when $M$ is finitely generated, the resolution has strong finiteness properties: precisely, it has finitely many non-zero linear strands and each linear strand (after a mild manipulation) admits the structure of a finitely generated $A$-module. 

The purpose of this paper is to quantify the result of \cite{symu1}, at least in part. Let $t_i(M)$ be the maximal degree of a generator of $F_i$; this is independent of the choice of minimal resolution. We note that $t_0(M)$ is the maximal degree of a generator of $M$, and $t_1(M)$ the maximal degree of an essential relation between the generators. We define the {\bf regularity} of $M$, denoted $\reg(M)$, to be the minimum integer $\rho$ such that $t_i(M) \le \rho+i$ for all $i \ge 0$. Our result from \cite{symu1} implies that $\reg(M)$ is finite when $M$ is finitely generated. Our main result here is the following theorem:

\begin{theorem} \label{thm:main}
Let $M$ be an $A$-module. Then $\reg(M)$ can be bounded by a function of $t_i(M)$ for $0 \le i \le 1+\lfloor \tfrac{1}{4} d^2 \rfloor$.
\end{theorem}

We note that $M$ is not required to be finitely generated in this theorem! Because of this, the theorem can easily be extended to allow for $A$-modules over arbitrary characteristic~0 coefficient rings, or even (twisted) $A$-modules over non-affine schemes, see Proposition~\ref{prop:reduction}.

\begin{remark}
Church and Ellenberg \cite{ce} proved this theorem for $d=1$ in the language of $\FI$-modules (and over any coefficient ring, with no restriction on characteristic). Their work directly inspired this paper. Church \cite{church} later greatly simplified the proof in \cite{ce}. We note that in characteristic~0, the $d=1$ case can be recovered from the results in \cite{symc1}.
\end{remark}

In fact, our results go beyond Theorem~\ref{thm:main} and give bounds on the regularities (and other invariants) of various local cohomology groups of $M$. We confine ourselves here to stating a consequence of these finer results. In \cite{symu1}, we showed that the Grothendieck group $\rK(A)$ is free of rank $2^d$ over the ring of symmetric functions $\Lambda$. In fact, we gave a canonical isomorphism
\begin{displaymath}
\rK(A) = \bigoplus_{r=0}^d \Lambda \otimes \rK(\Gr_r(E)).
\end{displaymath}
Given a symmetric function $f \in \Lambda$, let $\deg(f)$ be the maximum value of $\vert \lambda \vert$ over partitions $\lambda$ for which the coefficient of the Schur function $s_{\lambda}$ in $f$ is non-zero. Given a class $c \in \rK(A)$, write $c=\sum_{i \in I} f_i e_i$ where $\{e_i\}_{i \in I}$ is a basis for $\bigoplus_r \rK(\Gr_r(E))$ and $f_i \in \Lambda$, and define $\deg(c)$ to be the maximum value of $\deg(f_i)$. We show:

\begin{theorem} \label{thm:deg}
Let $M$ be a finitely generated $A$-module, and let $c=[M]$ be its class in $\rK(A)$. Then $\deg(c)$ can be bounded by a function of $t_i(M)$, for $i \le 1+\lfloor \tfrac{1}{4} d^2 \rfloor$.
\end{theorem}

As a corollary, we obtain similar bounds for any invariant of $A$-modules that factors through the Grothendieck group. For example, one can associate to $M$ a Hilbert series $\rH_M$, which is known to be rational \cite[Theorem~3.1]{delta}. Theorem~\ref{thm:deg} implies that the degree of the numerator of this series can be bounded from the $t_i(M)$ for $i \le 1+\lfloor \tfrac{1}{4} d^2 \rfloor$. (There are only finitely many possibilities for the denominator.)

\subsection{Applications}

We now give an application of Theorem~\ref{thm:main}. Let $M$ be an $A$-module. Treating $A$ and $M$ as Schur functors, let $A_n=A(\bC^n)$ and $M_n=M(\bC^n)$. Then $A_n$ is a polynomial ring in $nd$ variables and $M_n$ is a $\GL_n$-equivariant module over $A_n$. One can think of $(M_n)$ as a kind of compatible sequence of modules over the sequence of rings $(A_n)$, and it is interesting to study the asymptotic properties of these modules. Let $\reg(M_n)$ be the regularity of $M_n$ as an $A_n$-module. 

For the next statement, the invariant $\ell(M)$ is defined in \S\ref{s:bg}. Alternatively, one can use the inequality $\ell(M) \le t_0(M)+d$.

\begin{theorem} \label{thm:stable}
There is a function $F_d(\rho)$ depending only on $d$ and $\rho$ such that $\reg(M_n)$, for any $M$ and $n$, is bounded above by $F_d(\reg(M_k))$, where $k =  \ell(M) + \lfloor \tfrac{1}{4} d^2 \rfloor + d +1$.
\end{theorem}

\subsection{Open questions}

\begin{itemize}
\item Is the upper bound in Theorem~\ref{thm:main} optimal?
\item Assuming it is, what is the conceptual reason for this bound? Before proving the theorem, we thought a reasonable guess for the upper bound might be the Krull--Gabriel dimension of the category $\Mod_A$. But we showed in \cite{symu1} that this is $\binom{d+1}{2}$, which is (for $d$ large) more than twice as large as the bound in Theorem~\ref{thm:main}.
\item Can one give an explicit upper bound for $\reg(M)$ in terms of the relevant $t_i(M)$? With more work, a bound could be extracted from our proof, but it is likely far from optimal. At least for very small $d$, one could probably obtain reasonable bounds without too much additional effort. In the $d=1$ case, such a bound is given by Church and Ellenberg.
\item If $M$ is finitely generated, then $M$ has finite regularity, which implies that $n \mapsto \reg(M_n)$ is a weakly increasing function that is eventually constant, and equal to $\reg(M)$. For what value of $n$ does stabilization occur?
\item What is the relation between regularity and local cohomology for $A$-modules? For $d=1$ an explicit connection is established in \cite{regthm}.
\item Can Theorem~\ref{thm:main} be extended to positive characteristic? Note that in positive characteristic, it is not even known yet if regularity is finite.
\item Can we bound the regularity of each $\rR^i \Gamma_{\le r}(M)$ instead of just the complex $\rR \Gamma_{\le r}(M)$? (See Theorem~\ref{thm:main2} for context.)
\end{itemize}

\subsection{Outline}

In \S \ref{s:bg}, we recall background material from \cite{symu1} on the structure of $A$-modules. In \S \ref{s:prelim}, we introduce some invariants of $A$-modules and prove some basic results about them. In \S \ref{s:regbd}, we prove a number of inequalities involving regularity and the invariants introduced in \S \ref{s:prelim}, and eventually deduce the main theorem. Finally, in \S \ref{s:add}, we explain two additional results: an extension of the main theorem to arbitrary base schemes, and Theorem~\ref{thm:stable}.

\subsection{Notation}

We use notation as in \cite{symu1}. Throughout this paper, unless otherwise specified, $E$ is a $d$-dimensional complex vector space, $\bV = \bC^\infty$ is the vector representation of $\GL_\infty$, and $A$ is the tca $\bA(E) = \Sym(\bV \otimes E)$. We write $\cV$ for the category of polynomial representations of $\GL_\infty$. By $\Gr_r(E)$ we always mean the Grassmannian of $r$-dimensional quotients of $E$. We let $\cQ$ be the tautological rank $r$ quotient bundle on $\Gr_r(E)$. For a scheme $X$, we use the term ``$\cO_X$-module'' in place of ``quasi-coherent sheaf on $X$.''

\section{Background on tca's} \label{s:bg}

\subsection{Generalities on tca's}

Let $\bV=\bC^{\infty}=\bigcup_{n \ge 1} \bC^n$ be the standard representation of $\GL_{\infty}=\bigcup_{n \ge 1} \GL_n$. A representation of $\GL_{\infty}$ is {\bf polynomial} if it is a subquotient of a direct sum of tensor powers of $\bV$. We write $\cV$ for the category of polynomial representations. It is semi-simple abelian. The simple objects are the $\bS_{\lambda}(\bV)$, where $\bS_{\lambda}$ denotes the Schur functor associated to the partition $\lambda$. The category $\cV$ is closed under tensor products, and the tensor product of simple objects is computed with the Littlewood--Richardson rule.

Let $V \in \cV$. We say that a partition $\lambda$ {\bf occurs} in $V$ if the multiplicity of $\bS_{\lambda}(\bV)$ in $V$ is non-zero. We say that $V$ has {\bf finite degree} if there is an $N$ such that $\vert \lambda \vert < N$ for all $\lambda$ occurring in $V$. We let $\ell(\lambda)$ denote the number of rows in the partition $\lambda$, and define $\ell(V)$ to be the supremum of the $\ell(\lambda)$ over those $\lambda$ occurring in $V$.

For the purposes of this paper, a {\bf twisted commutative algbera} (tca) is a commutative algbera object in the tensor category $\cV$. Thus a tca is a  commutative associative unital $\bC$-algebra equipped with an action of $\GL_{\infty}$ by algebra automorphisms, under which it forms a polynomial representation. By a module over a tca $A$, we always mean a module object in $\cV$: such an object is a $\GL_{\infty}$-equivariant $A$-module for which the underlying $\GL_{\infty}$ representation is polynomial. The objects $A \otimes V$ with $V \in \cV$ are exactly the projective $A$-modules; we say that such a module has finite degree if $V$ does. We write $\Mod_A$ for the category of $A$-modules. It is a Grothendieck abelian category.

Let $E$ be a finite dimensional vector space of degree $d$. Then $\bA(E)=\Sym(E \otimes \bV)$ is a tca. In fact, it is the tca freely generated by $d$ elements of degree~1. These are the primary tca's of interest in this paper. We note that the category of $\bA(E)$-modules is locally noetherian (see, for example, \cite[Theorem~2.1]{symu1}). We also remark that the category of $A$-modules is equivalent to the category of $\FI_d$-modules \cite[Proposition 7.2.5]{catgb}.

More background on tca's can be found in \cite[\S 2]{symu1} and \cite{expos}.

\subsection{The structure of $A$-modules}

Fix $E$, and put $A=\bA(E)$. We briefly summarize the results from \cite{symu1} on the structure of $A$-modules. The reader should refer to \cite{symu1} for more complete statements.

Let $\fa_r \subset A$ be the $r$th determinantal ideal; it is generated by the representation $\lw^{r+1}(E) \otimes \lw^{r+1}(\bV) \subset A$. We let $\Mod_{A, \le r}$ be the subcategory of $\Mod_A$ consisting of modules supported on $V(\fa_r)$ (i.e., locally annihilated by a power of $\fa_r$). This defines a filtration of $\Mod_A$ that we call the {\bf rank stratification}. We define
\begin{displaymath}
\Mod_{A,>r} = \frac{\Mod_A}{\Mod_{A, \le r}}, \qquad \Mod_{A,r} = \frac{\Mod_{A,\le r}}{\Mod_{A,<r}},
\end{displaymath}
where the fraction notation indicates the Serre quotient. We sometimes write $\Mod_A^0$ in place of $\Mod_{A,\le 0}$ (the category of modules supported at~0) and $\Mod_A^{\gen}$ in place of $\Mod_{A, \ge d}$ (the generic category). We write
\begin{displaymath}
T_{>r} \colon \Mod_A \to \Mod_{A,>r}, \qquad S_{>r} \colon \Mod_{A,>r} \to \Mod_A
\end{displaymath}
for the localization functor and its right adjoint, the section functor, respectively. We note that $T_{>r}$ is exact and $S_{>r}$ is left-exact. We put $\Sigma_{>r}=S_{>r} \circ T_{>r}$ (the saturation functor). Finally, we let
\begin{displaymath}
\Gamma_{\le r} \colon \Mod_A \to \Mod_{A, \le r}
\end{displaymath}
be the functor assigning to an $A$-module $M$ the maximal submodule $\Gamma_{\le r}(M)$ of $M$ supported on $V(\fa_r)$. The general strategy employed in \cite{symu1} to understand $\Mod_A$ is to first understand the pieces $\Mod_{A,r}$, and then understand something about how $\Mod_A$ is built from them via the functors $\Gamma_{\le r}$ and $\Sigma_{\ge r}$.

We first study the category $\Mod_A^0$ \cite[\S 5.2]{symu1}, which is fairly easy to understand directly. All finitely generated modules in this category have finite length. The simple objects are just the irreducible polynomial representations of $\GL_{\infty}$, namely the $\bS_{\lambda}(\bV)$, with the maximal ideal of $A$ acting by~0. Thus the Grothendieck group $\rK(\Mod_A^0)$ is isomorphic to $\Lambda$, the ring of symmetric functions. Every finitely generated object has a finite length injective resolution, and the indecomposable injectives admit an explicit description \cite[Proposition~5.1]{symu1}.

We next examine the generic category \cite[\S 5.4]{symu1}. We show that $\Mod_A^{\gen}$ is equivalent to $\Mod_A^0$ \cite[Propositions~5.4,~5.6]{symu1}. In particular, all finitely generated objects have finite length and finite injective dimension. The indecomposable injective objects in $\Mod_A^{\gen}$ are exactly the localizations of the indecomposable projective objects of $\Mod_A$ (i.e., the modules $A \otimes \bS_{\lambda}(\bV)$) \cite[Corollary~5.7]{symu1}. The simple objects of $\Mod_A^{\gen}$ can be described as follows \cite[Corollary~5.7]{symu1}. There is a multiplication map $E \otimes \bV \otimes A \to A$, which induces a map $\bV \otimes A \to E^* \otimes A$. Let $\cK$ be the kernel of this map in $\Mod_A^{\gen}$. Then $\bS_{\lambda}(\cK)$ is simple, and every simple object has this form (for a unique $\lambda$).

Next, we consider the section functor $S_{\ge d} \colon \Mod_A^{\gen} \to \Mod_A$ \cite[\S 5.5]{symu1}. We connect this functor to the global sections functor on a Grassmannian (really an infinite dimensional one, though this is not said explicitly), and are thus able to compute the derived functors of $S_{\ge d}$ using Borel--Weil--Bott. With this perspective, we show that projective $A$-modules are derived saturated \cite[Corollary~5.17]{symu1}, that is, if $P$ is a projective $A$-module then the canonical map $P \to \Sigma_{\ge d}(P)$ is an isomorphism and $\rR^i \Sigma_{\ge d}(P)=0$ for $i>0$. We also compute the derived section functor on the simple objects $\bS_{\lambda}(\cK)$ \cite[Corollary~5.20]{symu1}.

We next study $\Mod_{A,r}$ for general $r$ \cite[\S 6.2]{symu1}. Every object in this category is locally annihilated by a power of $\fa_r$. We let $\Mod_{A,r}[\fa_r]$ be the full subcategory on objects annihilated by $\fa_r$. We give an explicit description of this subcategory in \cite[Theorem~6.5]{symu1}. Let $Y$ be the Grassmannian $\Gr_r(E)$, let $\cQ$ be the tautological rank $r$ bundle on $Y$, and let $B$ be the tca $\bA(\cQ)$ on $Y$. We show that $\Mod_{A,r}[\fa_r]$ is equivalent to $\Mod_B^{\gen}$. The results of the previous three paragraphs actually hold for tca's over general base schemes (with appropriate modifications) and can thus be applied to $\Mod_B^{\gen}$. For example, every object of $\Mod_B^{\gen}$ admits a filtration with graded pieces of the form $\cF \otimes \bS_{\lambda}(\cK)$, where $\cF$ is an $\cO_Y$-module and $\cK \in \Mod_B^{\gen}$ is as above. (This is very important for the present paper: we will have to understand the $\cO_Y$-modules $\cF$ quite well, in some respects.) In particular, we see that $\rK(\Mod_A)$ is isomorphic to $\Lambda \otimes \rK(Y)$, and thus free of rank $\binom{d}{r}$ as a $\Lambda$-module.

Next we study the section functor $S_{\ge r}$ on $\Mod_{A,r}[\fa_r]$. Let $S' \colon \Mod_B^{\gen} \to \Mod_B$ be the section functor for $B$ and let $\pi \colon Y \to \Spec(\bC)$ the structure map. Given $M \in \Mod_{A,r}[\fa_r]$ corresponding to $N \in \Mod_B^{\gen}$, we show \cite[Proposition~6.8]{symu1} that $\rR S_{\ge r}(M)$ is canonically isomorphic to $\rR \pi_* \rR S'(N)$. We have already explained how to compute $\rR S'$, and $\rR \pi_*$ is just usual sheaf cohomology on $Y$. This gives an effective way to compute $\rR S_{\ge r}$ on $\Mod_{A,r}[\fa_r]$, and is the primary tool for understanding $\rR S_{\ge r}$ in general.

Finally, we combine all of the above to prove our main theorems. We show \cite[Theorem~6.10]{symu1} that if $M$ is a finitely generated $A$-module then $\rR^i \Gamma_{\le r}(M)$ and $\rR^i \Gamma_{\ge r}(M)$ are finitely generated for all $i$ and $r$ and vanish for $i \gg 0$. As a consequence, we show that $\rD^b_{\fgen}(A)$, the bounded derived category with finitely generated cohomology, admits a semi-orthogonal decomposition into pieces equivalent to $\rD^b_{\fgen}(\Mod_{A,r})$. The functors $\rR \Gamma_{\le r}$ and $\rR \Sigma_{\ge r}$ should be thought of as truncating to the $\le r$ and $\ge r$ pieces in this decomposition. We define $\rR \Pi_r=\rR \Gamma_{\le r} \circ \rR \Sigma_{\ge r}$; it is effectively a projection onto the $r$th piece of the decomposition. As a corollary, we show that the Grothendieck group $\rK(A)$ is the direct sum of the Grothendieck groups $\rK(\Mod_{A,r})$, and thus free of rank $2^d$ as a $\Lambda$-module.

\section{Preliminary results} \label{s:prelim}

\subsection{General remarks on bounds}

In this section, $A$ is an arbitrary tca. By an {\bf invariant} of $A$-modules, we mean a rule $\alpha$ assigning to each $A$-module $M$ a quantity $\alpha(M) \in \bN \cup \{\infty\}$. We say that an invariant $\alpha$ satisfies the {\bf sum-sup condition} if for any collection $\{M_i\}_{i \in I}$ of $A$-modules we have
\begin{displaymath}
\alpha \big( \bigoplus M_i \big) = \sup_{i \in I}(\alpha(M_i)).
\end{displaymath}
We say that a module $M$ is {\bf $\alpha$-finite} if $\alpha(M)<\infty$. The following simple observation shapes the exposition of this paper:

\begin{proposition}
Let $\alpha$ and $\beta$ be two invariants satisfying the sum-sup condition. Then the following two statements are equivalent:
\begin{enumerate}[\rm \indent (a)]
\item Whenever $M$ is $\beta$-finite it is also $\alpha$-finite.
\item $\alpha$ is bounded by $\beta$; that is, there is a function $f \colon \bN \to \bN$ such that for any module $M$ with $\beta(M)$ finite we have $\alpha(M) \le f(\beta(M))$.
\end{enumerate}
\end{proposition}

\begin{proof}
It is clear that (b) implies (a). Suppose now that (b) does not hold. Then there is a collection of modules $\{M_i\}$ with $\beta(M_i)$ constant and $\alpha(M_i)$ tending to $\infty$. But then $\beta(\bigoplus M_i)$ is finite and $\alpha(\bigoplus M_i)$ is infinite. Thus (a) does not hold.
\end{proof}

Note that in the proposition it is essential that we allow modules that are not finitely generated. In this paper, we prove that various invariants are bounded by other invariants. We always phrase our results as in the (a) statement above. This leads to a cleaner exposition since we do not have to produce the bounding function $f$. All of our proofs are in fact effective, and with additional work one could produce the bounding functions.

We make one more simple observation:

\begin{proposition} \label{prop:sumsupfin}
Let $\alpha$ be an invariant satisfying the sum-sup condition and let $F$ be an endofunctor of $\Mod_A$ that commutes with arbitrary direct sums. Suppose that $F(M)$ is $\alpha$-finite whenever $M$ is finitely generated. Let $V$ be a representation of $\GL_{\infty}$ of finite degree. Then $F(M \otimes V)$ is also $\alpha$-finite.
\end{proposition}

\begin{proof}
Replacing $\alpha$ with $\alpha \circ F$, we can assume $F$ is the identity. Suppose $V$ has degree $\le n$. Then $V=\bigoplus_{\vert \lambda \vert \le n} \bS_{\lambda}(\bV)^{\oplus m(\lambda)}$ for some (possibly infinite) cardinal $m(\lambda)$. Since $\alpha$ satisfies sum-sup, we thus see that $\alpha(M \otimes V)$ is bounded by the maximum of the $\alpha(M \otimes \bS_{\lambda}(\bV))$ for $\vert \lambda \vert \le n$. Each of these $\alpha$'s is finite by assumption, and there are only finitely many of them.
\end{proof}

\begin{remark}
The definitions and observations in this section are not specific to $A$-modules, and can be applied in much greater generality. We will also apply them to the invariant $\rho$ introduced in \S \ref{ss:grassreg}.
\end{remark}

\subsection{Some filtrations}

For a partition $\lambda=(\lambda_1, \lambda_2, \ldots)$, let $t_r(\lambda)=(\lambda_{r+1}, \lambda_{r+2}, \ldots)$ and put $\tau_r(\lambda) = \vert t_r(\lambda) \vert$. Thus $t_r(\lambda)$ discards the first $r$ rows of $\lambda$, and $\tau_r(\lambda)$ counts how many boxes $\lambda$ has below the $r$th row. For an object $V$ of $\cV$ we let $\tau_r(V)$ be the supremum of the $\tau_r(\lambda)$'s over those $\lambda$'s occurring in $V$, with the convention $\tau_r(V)=-\infty$ if $V=0$. We note that if $W \subset V$ then $\tau_r(V)$ is equal to the maximum of $\tau_r(W)$ and $\tau_r(V/W)$. Much of this paper is concerned with bounding $\tau_r(M)$, for $M$ an $A$-module.

Let $V$ be an object of $\cV$. We define $F^r_nV$ (resp.\ $\gr^r_n(V)$, $\gr^r_{\mu}(V)$) to be the sum of the $\lambda$-isotypic pieces of $V$ with $\tau_r(\lambda) \ge n$ (resp.\ $\tau_r(\lambda)=n$, $t_r(\lambda)=\mu$). We have a descending filtration on $V$:
\begin{displaymath}
\cdots \subset F^r_2V \subset F^r_1V \subset F^r_0V = V,
\end{displaymath}
and $\gr^r_n(V)$ is identified with the associated graded $F^r_nV/F^r_{n+1}V$. We further have a decomposition
\begin{equation} \label{eq:grdecomp}
\gr^r_n(V) = \bigoplus_{\vert \mu \vert=n} \gr^r_{\mu}(V)
\end{equation}
Since the constructions $F^r_nV$, $\gr^r_n(V)$, and $\gr^r_{\mu}(V)$ are all summands of $V$, they are exact functors of $V$. We note that $\tau_r(V)$ is the minimum $n$ such that $F^r_{n+1}(V)=0$ or $\gr^r_m(V)=0$ for all $m>n$.

Suppose now that $M$ is an $A$-module (now $A=\bA(E)$). Then $F^r_nM$ is an $A$-submodule of $M$ and so $\gr^r_n(M)$ is also an $A$-module. The Pieri rule shows that $\fa_r F^r_n M \subset F^r_{n+1} M$, and so $\gr^r_n(M)$ is annihilated by $\fa_r$. (We remark that $F^r_{\bullet} M$ is \emph{not} the $\fa_r$-adic filtration, in general.) The Pieri rule also shows that $\gr^r_{\mu}(M)$ is an $A$-submodule of $\gr^r_n(M)$, and so the decomposition in \eqref{eq:grdecomp} with $V=M$ is one of $A$-modules.

If $M$ is supported on $V(\fa_{r-1})$ then so is $F^r_nM$, since this is a submodule of $M$. It follows that the functors $F^r_n$, $\gr^r_n$, and $\gr^r_{\mu}$ on $\Mod_A$ descend to functors on $\Mod_{A,\ge r}$ which we denote with the same notation. For $M \in \Mod_{A,\ge r}$, we define $\sigma_r(M)$ to be the maximum $n$ such that $F^r_n M \ne 0$, or $\infty$ if no such $n$ exists; we also put $\sigma_r(M)=-\infty$ if $M=0$. For an $A$-module $M$ we put $\sigma_r(M)=\sigma_r(T_{\ge r}(M))$. Clearly, we have $\sigma_r(M) \le \tau_r(M)$.

\begin{proposition} \label{prop:invt-ann}
Let $M$ be an $A$-module.
\begin{enumerate}[\indent \rm (a)]
\item If $k=\tau_r(M)$ is finite then $M$ is annihilated by $\fa_r^{k+1}$.
\item If $k=\sigma_r(M)$ is finite then $\fa_r^{k+1} M$ is supported on $V(\fa_{r-1})$.
\item If $M$ is finitely generated and supported on $V(\fa_r)$ then $\tau_r(M)$ and $\sigma_r(M)$ are finite.
\end{enumerate}
\end{proposition}

\begin{proof}
(a) Since $\fa_r F^r_n M \subset F^r_{n+1} M$ and $F^r_{k+1}M=0$, it follows that $\fa_r^{k+1}M=0$.

(b) As in (a), we have $T_{\ge r}(\fa_r^{k+1} M)=0$, which implies the statement.

(c) Since $M$ is finitely generated and supported on $V(\fa_r)$, it has a finite length filtration where the graded pieces are annihilated by $\fa_r$, and it suffices to show that $\tau_r$ of each graded piece is finite. We may thus assume $M$ is annihilated by $\fa_r$. Thus $M$ is a quotient of $A/\fa_r \otimes V$ for some $V \in \cV^{\fgen}$. But $\ell(A/\fa_r)=r$, and so $\tau_r(A/\fa_r \otimes V)$ is equal to the maximum size of a partition occurring in $V$, which is finite. It follows that $\tau_r(M)$ is finite as well.
\end{proof}

For a complex $M$ in $\cV$, we define $\tau_r(M)$ to be the supremum of $\tau_r(\rH^i(M))$ over $i \in \bZ$, and similarly for $\sigma_r(M)$. In particular, this definition applies to objects of the derived category of $A$-modules.

\subsection{Construction of sheaves}

Let $Y=\Gr_r(E)$ with tautological bundle $\cQ$ and structure map $\pi \colon Y \to \Spec(\bC)$. Let $B=\bA(\cQ)$ and let $\cK$ be the kernel of the canonical map $\bV \otimes B \to \cQ^* \otimes B$ in $\Mod_B^{\gen}$. Let
\begin{displaymath}
\Psi \colon \Mod_{A,r}[\fa_r] \to \Mod_B^{\gen}
\end{displaymath}
be the equivalence constructed in \cite[\S 6.2]{symu1}

Suppose that $M \in \Mod_{A,r}$ is finitely generated. Then the $\fa_r$-adic filtration on $M$ is finite. The graded pieces belong to $\Mod_{A,r}[\fa_r]$ and so, by the results of \cite{symu1}, can be  filtered such that the graded pieces correspond under $\Psi$ to objects of the form $\cF \otimes \bS_{\lambda}(\cK)$, where $\cF$ is an $\cO_Y$-module. In this section, we explain how to construct the sheaves $\cF$ directly from $M$.

For a partition $\mu$ and an object $V$ of $\cV$, put
\begin{displaymath}
\bM^r_{\mu}(V) = \bigoplus_{k \ge \mu_1} V_{[k^r,\mu]}.
\end{displaymath}
The coordinate ring of $\Gr_r(E)$ under the Pl\"ucker embedding is
\begin{displaymath}
\bM^r_{\emptyset}(A) = \bigoplus_{k \ge 0} \bS_{(k^r)}(E).
\end{displaymath}
If $M$ is an $A$-module, we have a map on isotypic components
\begin{displaymath}
(\bS_{(n^r)}(\bV) \otimes \bS_{(n^r)}(E)) \otimes (\bS_{[m^r,\mu]}(\bV) \otimes M_{[m^r,\mu]})
\to \bS_{[(n+m)^r,\mu]}(\bV) \otimes M_{[(n+m)^r,\mu]}.
\end{displaymath}
By taking the $\bS_{[(n+m)^r, \mu]}(\bV)$-isotypic component, this induces a map
\begin{displaymath}
\bS_{(n^r)}(E) \otimes M_{[m^r,\mu]} \to M_{[(n+m)^r,\mu]},
\end{displaymath}
and one verifies that this gives $\bM^r_{\mu}(M)$ the structure of an $\bM^r_{\emptyset}(A)$-module.

\begin{definition}
We let $\sh^r_{\mu}(M)$ be the sheaf on $\Gr_r(E)$ obtained from $\bM^r_{\mu}(M)$ by applying the tilde construction.
\end{definition}

Since $\bM^r_{\mu}(V)$ is an exact functor of $V$ and the tilde construction is exact, it follows that $\sh^r_{\mu}(M)$ is an exact functor of $M$. If $M$ is finitely generated and supported on $V(\fa_{r-1})$ then $\tau_{r-1}(M)<\infty$, and so $\bM^r_{\mu}(M)_n=0$ for $n \gg 0$, which implies that $\sh^r_{\mu}(M)=0$. Since $\sh^r_{\mu}(M)$ commutes with direct limits, it follows that $\sh^r_{\mu}$ kills all modules supported on $V(\fa_{r-1})$. We thus see that $\sh^r_{\mu}$ factors through $\Mod_{A,\ge r}$. We use the same notation for the induced functor on the quotient.

\begin{proposition} \label{prop:gr}
Let $M \in \Mod_{A,r}[\fa_r]$ correspond to $\cF \otimes \bS_{\mu}(\cK)$ in $\Mod_B^{\gen}$ under $\Psi$. Then
\begin{displaymath}
\gr^r_{\lambda}(M)=\begin{cases} M & \text{if $\lambda=\mu$} \\
0 & \text{if $\lambda \ne \mu$} \end{cases}, \qquad
\sh^r_{\lambda}(M)=\begin{cases} \cF & \text{if $\lambda=\mu$} \\
0 & \text{if $\lambda \ne \mu$} \end{cases}.
\end{displaymath}
\end{proposition}

\begin{proof}
\addtocounter{equation}{-1}
\begin{subequations}
Let $S' \colon \Mod_B^{\gen} \to \Mod_B$ be the section functor. By \cite[Corollary~5.20]{symu1}, we have
\begin{displaymath}
S'(\cF \otimes \bS_{\mu}(\cK)) = \cF \otimes \bigoplus_{\substack{\lambda\\ \lambda_d \ge \mu_1}} \bS_{[\lambda,\mu]}(\bV) \otimes \bS_{\lambda}(\cQ),
\end{displaymath}
(note that $\cF$ pulls out of $S'$). Applying $\pi_*$ to this, and using the identification $S_{\ge r}=\pi_* \circ S'$ \cite[Proposition~6.8]{symu1}, we find
\begin{equation} \label{eq:satid}
S_{\ge r}(M) = \bigoplus_{\substack{\lambda\\ \lambda_d \ge \mu_1}} \bS_{[\lambda,\mu]}(\bV) \otimes \rH^0(Y, \cF \otimes \bS_{\lambda}(\cQ)).
\end{equation}
Thus $\gr^r_{\lambda}(S_{\ge r}(M))=0$ if $\lambda \ne \mu$. By definition, $\gr^r_{\lambda}(M)$ is the image of $\gr^r_{\lambda}(S_{\ge r}(M))$ in $\Mod_{A,r}$, and so we see that $\gr^r_{\lambda}(M)=0$ if $\lambda \ne \mu$. We thus see $F^r_nM=M$ for $n \le \vert \mu \vert$ and $F^r_nM=0$ for $n>\vert \mu \vert$. It follows that $\gr^r_n(M)=M$, and so, appealing to \eqref{eq:grdecomp}, we find $\gr^r_{\mu}(M)=M$, since all the other pieces in the sum vanish.

From \eqref{eq:satid}, it follows that $\bM^r_{\lambda}(S_{\ge r}(M))=0$ if $\lambda \ne \mu$. Thus $\sh^r_{\lambda}(M)=0$ if $\lambda \ne \mu$. From \eqref{eq:satid}, we have
\begin{displaymath}
\bM^r_{\mu}(S_{\ge r}(M)) = \bigoplus_{k \ge \mu_1} \rH^0(Y, \cF(k)),
\end{displaymath}
where $\cF(k)$ denotes the $k$th twist of $\cF$ by $\cO_Y(1)=\bS_{1^r}(\cQ)=\det(\cQ)$. Applying the tilde construction to this gives $\cF$, and so $\sh^r_{\mu}(M)=\cF$.
\end{subequations}
\end{proof}

\begin{corollary} \label{cor:sh-K}
Let $M$ be an object of $\Mod_{A,r}$. Then $\gr^r_{\mu}(M) \in \Mod_{A,r}[\fa_r]$ corresponds to $\sh^r_{\mu}(M) \otimes \bS_{\mu}(\cK) \in \Mod_B^{\gen}$ under $\Psi$.
\end{corollary}

\begin{proof}
Replacing $M$ with $\gr^r_{\mu}(M)$, we may as well assume $M=\gr^r_{\mu}(M)$. By the proposition, we see that $M=\gr^r_\mu(M)$ only happens if $\Psi(M)$ has the form $\cF \otimes \bS_{\mu}(\cK)$ for some $\cF$. Furthermore, the proposition implies $\cF=\sh^r_{\mu}(M)$, which completes the proof.
\end{proof}

\subsection{The $\rho$ invariant} \label{ss:grassreg}

As in the previous section, let $Y=\Gr_r(E)$, let $\pi \colon Y \to \Spec(\bC)$ be the structure map, and let $\cQ$ be the tautological quotient bundle. Given an $\cO_Y$-module $\cF$, we define $\rho(\cF)$ to be the minimal non-negative integer $\rho$ such that $\rH^i(Y, \cF \otimes \bS_{\lambda}(\cQ))=0$ whenever $i>0$ and $\lambda_r \ge \rho$. Note that 
\[
\bS_{(\lambda_1+1, \dots, \lambda_r+1)}(\cQ) \cong \cO(1) \otimes \bS_\lambda(\cQ)
\]
where $\cO(1)$ is the ample line bundle corresponding to the Pl\"ucker embedding. Thus $\rho$ is related to the notion of Castelnuovo--Mumford regularity, though it is not the same. We note that $\rho$ satisfies the sum-sup condition, since cohomology on $Y$ commutes with filtered direct limits \cite[Tag 07TB]{stacks}.

\begin{proposition} \label{prop:coh}
Let $X$ be a projective variety of dimension $n$, let $\cM$ be an $\cO_X$-module, and let
\begin{displaymath}
\cF_{n-1} \to \cF_{n-2} \to \cdots \to \cF_0 \to \cM \to 0,
\end{displaymath}
be a partial resolution. Suppose that each $\cF_i$ has no higher cohomology. Then $\cM$ has no higher cohomology.
\end{proposition}

\begin{proof}
Let $\cG_i$ be the cokernel of $\cF_{i+1} \to \cF_i$ (with the convention $\cF_n=0$). We prove that $\rH^j(X, \cG_i) = 0$ for $j > i$ by descending induction on $i$. If $i=n-1$, then we have a surjection $\cF_{n-1} \to \cG_{n-1} \to 0$ and $\rH^n(X, -)$ is right-exact (since $n=\dim{X}$), so $\rH^n(X, \cG_{n-1}) = 0$. 

In general, we have the short exact sequence $0 \to \cG_{i+1} \to \cF_i \to \cG_i \to 0$ which gives
\begin{displaymath}
\rH^j(Y, \cF_i) \to \rH^j(Y, \cG_i) \to \rH^{j+1}(Y, \cG_{i+1}).
\end{displaymath}
If $j > i$, then the rightmost term vanishes by induction, and the leftmost term vanishes by assumption, so we conclude that $\rH^j(X, \cG_i) = 0$. Since $\cM=\cG_0$, we see that $\cM$ has no higher cohomology.
\end{proof}

\begin{proposition} \label{prop:regres}
Let $n=r(d-r)$ be the dimension of $Y$. Given a partial resolution
\begin{displaymath}
\cF_{n-1} \to \cF_{n-2} \to \cdots \to \cF_0 \to \cM \to 0,
\end{displaymath}
we have $\rho(\cM) \le \max \rho(\cF_i)$. In particular, if each $\cF_i$ is $\rho$-finite then so is $\cM$.
\end{proposition}

\begin{proof}
Let $\rho=\max \rho(\cF_i)$ and let $\lambda$ be a partition with $\lambda_r \ge \rho$. Then, by assumption $\cF_i \otimes \bS_{\lambda}(\cQ)$ has no higher cohomology, for all $i$. Thus, by Proposition~\ref{prop:coh}, we see that $\cM \otimes \bS_{\lambda}(\cQ)$ has no higher cohomology, and so $\rho(\cM) \le \rho$.
\end{proof}

\section{Regularity bounds} \label{s:regbd}

\subsection{Regularity of complexes}

Given $M \in \rD^b(A)$ and $n \in \bZ$, we define $\reg_n(M)$ to be the infimum of non-negative integer $\rho$ such that $\rH^{-i}(M \otimes^{\rL}_A \bC)$ is supported in degrees $\le i+\rho$ for all $i \le n$, including negative values of $i$. We define $\reg(M)$, the {\bf regularity} of $M$, to be the supremum of the $\reg_n(M)$ over all $n$. The regularity does not provide any lower bound on the degrees in $\rH^{-i}(M \otimes^{\rL}_A \bC)$. However, for most complexes we care about, there is a trivial lower bound. Indeed, let $\rD^{\ge 0, b}(A)$ be the subcategory of the derived category $\rD(A)$ on those complexes $M$ such that $\rH^n(M)=0$ for $n<0$ and $n \gg 0$.  Then for $M \in \rD^{\ge 0,b}(A)$, the module $\rH^{-i}(M \otimes^{\rL}_A \bC)$ is always supported in degrees $\ge i$.

\begin{proposition}
Let $M \in \rD^b(A)$. Let $k$ be maximal so that $\rH^k(M) \ne 0$. Then $\reg_n(M) \ge k$ for all $n \ge 0$.
\end{proposition}

\begin{proof}
We have $\rH^k(M \otimes^{\rL}_A \bC)=\rH^k(M) \otimes_A \bC$. Since $\rH^k(M)$ is non-zero, so is $\rH^k(M) \otimes_A \bC$ by Nakayama's lemma. Suppose $\reg_n(M)=\rho$. Since $\rH^k(M \otimes^{\rL}_A \bC)$ is supported in some non-negative degree, we must have $0 \le -k+\rho$, so $\rho \ge k$.
\end{proof}

Thus any bound on regularity also bounds the number of non-zero cohomology groups. The following proposition is a useful way to reduce statements about complexes to modules.

\begin{proposition} \label{prop:concentrate}
Let $M \in \rD^{\ge 0, b}(A)$. Suppose $\reg_{-1}(M)$ is finite. Then there is a triangle
\begin{displaymath}
F \to M \to N \to
\end{displaymath}
where $F$ is a complex of free modules such that $F^n=0$ for $n \le 0$ and $F^n=0$ for $n \gg 0$ and each $F^i$ has finite degree, and $N$ is an $A$-module, regarded as a complex in degree~$0$.
\end{proposition}

\begin{proof}
Let $\cdots \to M^{-1} \to M^0 \to M^1 \to \cdots$ be a complex consisting of free modules which represents $M$. We take $N = \cdots \to M^{-2} \to M^{-1}$; by the condition that $M \in \rD^{\ge 0,b}(A)$, $N$ can be represented by $\coker(M^{-2} \to M^{-1})$, a module concentrated in degree $0$. Now consider the subcomplex $M^0 \to M^1 \to \cdots$. This is a direct sum of a minimal complex (one in which the differentials become $0$ upon tensoring with $\bC$) and trivial complexes (two term complexes consisting of an isomorphism between free modules). So $M$ is equivalent to a minimal complex that we call $F$. 

By our assumptions, $\rH^n(M) = 0$ for $n \ge D$ for some bound $D$, and all cohomology groups are concentrated in finitely many degrees. Let $d$ be the maximal degree appearing; the subcomplex of $F$ consisting of free modules generated in degree $>d$ is equivalent to $0$ by Nakayama's lemma, and hence by our minimality assumption, it is zero, so $F$ is concentrated in degrees $\le d$. In particular, the maximal degree of a generator of $F^n$ becomes a strictly decreasing function when $n \ge D$, so we see that $F^n = 0$ for $n \gg 0$.
\end{proof}

\begin{remark}
In the setting of the above proposition, it is easy to translate bounds on regularity between $N$ and $M$. For instance, if $\reg_n(M)$ is finite then $\reg_n(N)$ is finite, and if $\reg(N)$ is finite then $\reg(M)$ is finite.
\end{remark}

\begin{proposition} \label{prop:taud}
Let $M \in \rD^{\ge 0, b}(A)$. Suppose that $\reg_0(M)$ is finite. Then $\tau_d(M)$ is finite.
\end{proposition}

\begin{proof}
First suppose that $M$ is a module. Since $\reg_0(M)$ is finite, it follows that $M$ is a quotient of $V \otimes A$ for some finite degree $V \in \cV$. Since $A = \Sym(\bV)^{\otimes d}$, we have $\ell(A) = d$ by the Pieri rule. Another application of the Pieri rule implies that $\tau_d(\bS_\lambda(\bV) \otimes A)=\vert \lambda \vert$. In particular, $\tau_d(V \otimes A)$ is bounded by the degree of $V$, and thus finite. It follows that $\tau_d(M)$ is finite as well.

We now treat the general case. Let $F \to M \to N \to$ be a triangle as in Proposition~\ref{prop:concentrate}, so that $F$ is a finite length complex of finite degree free modules and $N$ is a module. Since $\reg(F)$ is finite and $\reg_0(M)$ is finite, it follows that $\reg_0(N)$ is finite. Thus, by the previous paragraph, $\tau_d(N)$ is finite. Since $\tau_d(F)$ is finite, it follows that $\tau_d(M)$ is finite.
\end{proof}

\subsection{Regularity and change of rings}

For an $A/\fa_r^k$-module $M$, we write $\reg_n(M; A/\fa_r^k)$ for the $\reg_n$ of $M$ over $A/\fa_r^k$; this is defined just like $\reg_n(M)$, but using Tor's over $A/\fa_r^k$ instead of over $A$. We now examine how this quantity relates to $\reg_n(M)$.

\begin{proposition} \label{prop:cmpreg}
Let $A=\bA(E)$ and suppose $M$ is an $A$-module annihilated by $\fa_r^k$. Then $\reg_n(M; A/\fa_r^k)$ is finite if and only if $\reg_n(M;A)$ is finite.
\end{proposition}

\begin{proof}
We prove this by induction on $n$. When $n=0$, we have $\reg_0(M; A/\fa_r^k) = \reg_0(M;A)$ since this just measures the maximum degree of a minimal generator, which is unaffected by which algebra we consider $M$ to be a module over.

Now suppose that $\reg_0(M;A)=\reg_0(M;A/\fa_r^k)$ is finite; if not, then $\reg_n$ is infinite for both algebras and there is nothing to show. Then we have a short exact sequence
\[
0 \to N \to V \otimes A/\fa_r^k \to M \to 0
\]
for some finite degree $\GL_\infty$ representation $V$, and hence an exact sequence
\[
\Tor_n^A(V \otimes A/\fa_r^k, \bC) \to \Tor_n^A(M, \bC) \to \Tor_{n-1}^A(N, \bC) \to \Tor_{n-1}^A(V \otimes A/\fa_r^k, \bC).
\]
By noetherianity of $A$, $\reg_n(A/\fa_r^k;A)$ is finite for all $n$, and hence the same is true if we tensor by $V$. In particular, $\Tor_n^A(M, C)$ is concentrated in finitely many degrees if and only if the same is true for $\Tor_{n-1}^A(N,\bC)$. By induction, we know that $\reg_{n-1}(N;A)$ is finite if and only if $\reg_{n-1}(N;A/\fa_r^k)$ is finite, so we conclude that $\reg_n(M;A)$ is finite if and only if $\reg_n(M;A/\fa_r^k)$ is finite.
\end{proof}

\subsection{Bounding $\rho$ from regularity}

For $M \in \Mod_{A,r}$, define $\rho_r(M)$ to be the supremum of the $\rho(\sh^r_{\lambda}(M))$, over all $\lambda$. For an $A$-module $M$ supported on $V(\fa_r)$, define $\rho_r(M)$ to be $\rho_r(T_{\ge r}(M))$.

\begin{proposition} \label{prop:regsh}
Let $M$ be an $A$-module annihilated by $\fa_r^k$ for some $k$, and let $n=r(d-r)$ be the dimension of $\Gr_r(E)$. Suppose $\reg_{n-1}(M)$ is finite. Then $\rho_r(M)$ is also finite.
\end{proposition}

\begin{proof}
By Proposition~\ref{prop:cmpreg} we have that $\reg_{n-1}(M; A/\fa_r^k)$ is finite. Let $A/\fa_r^k \otimes V_{\bullet} \to M$ be a partial resolution of length $n$ with each $V_i$ of finite degree. Since $\sh^r_{\lambda}$ is exact, we have a partial resolution $\sh^r_{\lambda}(A/\fa_r^k \otimes V_{\bullet}) \to \sh^r_{\lambda}(M)$. By  Proposition~\ref{prop:sumsupfin}, $\rho(\sh^r_{\lambda}(A/\fa_r^k \otimes V_i))$ is finite for each $i$. Thus $\rho(\sh^r_{\lambda}(M)$) is finite by Proposition~\ref{prop:regres}.
\end{proof}

\begin{corollary}
Let $M$ be an $A$-module. Suppose that $\tau_r(M)$ and $\reg_{n-1}(M)$ are finite, with $n=r(d-r)$. Then $\rho_r(M)$ is finite.
\end{corollary}

\subsection{Bounds on $\tau$ of local cohomology}

The goal of this section is Proposition~\ref{prop:taubd}, which proves $\tau_r$-finiteness of $\rR \Gamma_{\le r}(M)$ under certain conditions.

\begin{proposition}
Let $M \in \Mod_{A,r}$. For all $n>0$, we have
\begin{displaymath}
\tau_{r-1}(\rR^n S_{\ge r}(M)) \le \sigma_r(M) + \max(\sigma_r(M), \rho_r(M)).
\end{displaymath}
\end{proposition}

\begin{proof}
First suppose that $M=\gr^r_{\lambda}(M)$ for some $\lambda$. Let $\cF=\sh^r_{\lambda}(M)$, so that under the equivalence $\Mod_{A,r}[\fa_r]=\Mod_B^{\gen}$ the object $M$ corresponds to $\cF \otimes \bS_{\lambda}(\cK)$ (Corollary~\ref{cor:sh-K}). Recall that $\rR S_{\ge r}$ on $\Mod_{A,r}[\fa_r]$ corresponds to $\rR \pi_* \circ \rR S'$ on $\Mod_B^{\gen}$, where $S' \colon \Mod_B^{\gen} \to \Mod_B$ is the section functor \cite[Proposition~6.8]{symu1}. By \cite[Corollary~5.20]{symu1} we have
\begin{displaymath}
\rR^j S'(\cF \otimes \bS_{\lambda}(\cK))=\cF \otimes \bigoplus_{\mu} \bS_{[\mu,\lambda]_j}(\bV) \otimes \bS_{\mu}(\cQ),
\end{displaymath}
where $\bS_{[\lambda,\mu]_j}(\bV)$ means $\bS_{\nu}(\bV)$ if Bott's algorithm terminates in $j$ steps to $\nu$, and means~0 otherwise. We thus obtain a spectral sequence
\begin{displaymath}
\bigoplus_{\mu} \rR^i \pi_*(\cF \otimes \bS_{\mu}(\cQ)) \otimes \bS_{[\mu,\lambda]_j}(\bV) \implies \rR^{i+j} S_{\ge r}(\cF \otimes \bS_{\lambda}(\cK)).
\end{displaymath}
First suppose $j>0$. Then $\bS_{[\mu,\lambda]_j}(\bV) = 0$ unless $\mu_r<\lambda_1$, and so
\begin{displaymath}
\tau_{r-1}([\mu,\lambda]_j) \le \vert \lambda \vert + \lambda_1 \le 2 \vert \lambda \vert.
\end{displaymath}
Now suppose $j=0$ and $i>0$. Then $\rR^i \pi_*(\cF \otimes \bS_{\mu}(\cQ))=0$ if $\mu_r \ge \rho(\cF)$. We thus only get contributions $\bS_{[\mu,\lambda]}(\bV)$ with $\mu_r<\rho(\cF)$. Thus $\tau_{r-1}([\mu,\lambda]) \le \rho(\cF)+\vert \lambda \vert$. Thus, counting the contributions from $j=0$ and $j>0$, we find
\begin{displaymath}
\tau_{r-1}(\rR^n S_{r-1}(M)) \le \max(2\vert \lambda \vert, \vert \lambda \vert + \rho(\cF)) = \vert \lambda \vert + \max(\vert \lambda \vert, \rho(\cF))
\end{displaymath}
which proves the proposition in this case. (Note $\sigma_r(M)=\vert \lambda \vert$.)

We now consider the general case. We have the short exact sequence
\begin{displaymath}
0 \to M' \to M \to \gr_n(M) \to 0,
\end{displaymath}
and thus an exact sequence
\begin{displaymath}
\rR^i S_{\ge r}(M') \to \rR^i S_{\ge r}(M) \to \rR^i S_{\ge r}(\gr_n(M)).
\end{displaymath}
Let $a=\sigma_r(M)$ and $b=\max_{\lambda} \rho(\sh^r_{\lambda}(M))$. Then $\sigma_r$ of $M'$ and $\gr_n(M)$ are both bounded by $a$, while the maximum regularity of $\sh^r_{\lambda}$ for $M'$ and $\gr_n(M)$ are bounded by $b$. Thus by induction, $\tau_{r-1}$ of the left term above is bounded by $a+\max(a,b)$. On the other hand, $\tau_{r-1}$ of the right term is bounded by $a+\max(a,b)$ by the previous paragraph. The result thus follows.
\end{proof}

\begin{corollary} \label{cor:taubd}
Let $M$ be an $A$-module. Suppose that $\tau_r(M)$ and $\rho_r(M)$ are finite. Then $\tau_{r-1}(\rR^i \Sigma_{\ge r}(M))$ is finite for all $i>0$.
\end{corollary}

\begin{proposition} \label{prop:taubd2}
Let $M$ be an $A$-module. Fix $r$ and let $n=r(d-r)$. Suppose that $\tau_r(M)$ and $\reg_{n+1}(M)$ are finite. Then $\tau_{r-1}(\rR \Gamma_{<r}(M))$ is finite.
\end{proposition}

\begin{proof}
We actually prove the following more precise statements:
\begin{enumerate}
\item If $\tau_r(M)$ and $\reg_{n-1}(M)$ are finite then so is $\tau_{r-1}(\rR^i \Gamma_{<r}(M))$ for all $i \ge 2$.
\item If $\tau_r(M)$ and $\reg_n(M)$ are finite then so is $\tau_{r-1}(\rR^1 \Gamma_{<r}(M))$.
\item If $\tau_r(M)$ and $\reg_{n+1}(M)$ are finite then so is $\tau_{r-1}(\Gamma_{<r}(M))$.
\end{enumerate}
We begin with (a). Since $\rR^i \Gamma_{<r}(M)=\rR^{i-1}\Sigma_{\ge r}(M)$ for $i \ge 2$, the result follows from the previous corollary ($\rho_r$ can be bounded from $\reg_{n-1}$ by Proposition~\ref{prop:regsh}).

We now prove (b). Consider an exact sequence
\begin{displaymath}
0 \to N \to A/\fa_r^k \otimes V \to M \to 0
\end{displaymath}
with $V$ of finite degree. Since $\reg_n(M)$ is finite, so is $\reg_{n-1}(N)$. We have an exact sequence
\begin{displaymath}
\rR^1 \Gamma_{<r}(A/\fa_r^k \otimes V) \to \rR^1 \Gamma_{<r}(M) \to \rR^2 \Gamma_{<r}(N)
\end{displaymath}
The left term has finite $\tau_{r-1}$ (Proposition~\ref{prop:sumsupfin}). The right term has finite $\tau_{r-1}$ by part (a). Thus the middle term has finite $\tau_{r-1}$.

We now prove (c). Consider the same exact sequence. Since we now assume $\reg_{n+1}(M;A)$ is finite, we have $\reg_n(N)$ finite. Since $\Gamma_{<r}(A/\fa_r^k)=0$, we have an injection $\Gamma_{<r}(M) \to \rR^1 \Gamma_{<r}(N)$. Since the target has finite $\tau_{r-1}$ by (b), the source has finite $\tau_{r-1}$.
\end{proof}

We now generalize Proposition~\ref{prop:taubd2} to complexes. Recall that $F^r_nM$ is the sum of the $\lambda$-isotypic pieces of $M$ with $\tau_r(\lambda) \ge n$, and this is a submodule of $M$. We define $E^r_nM$ to be the quotient module $M/F^r_{n+1} M$. This is the maximal quotient having $\tau_r \le n$. We note that this only depends on the underlying $\GL_{\infty}$-representation, and so $E^r_n$ is an exact functor. The functor $E^r_n$ is only used in the following proof.

\begin{proposition} \label{prop:taubd}
Let $M \in \rD^{\ge 0,b}(A)$. Suppose that $\tau_r(M)$ and $\reg_{n+1}(M)$ are finite, where $n=r(d-r)$. Then $\rR \Gamma_{<r}(M)$ is $\tau_{r-1}$-finite.
\end{proposition}

\begin{proof}
Let $k=\tau_r(M)$. Pick a triangle
\begin{displaymath}
F \to M \to N \to
\end{displaymath}
as in Proposition~\ref{prop:concentrate}. Apply $E^r_k$ to the above triangle. Since this is exact, it is well-defined on the derived category and commutes with cohomology. In particular, the natural map $M \to E^r_k(M)$ is an isomorphism in the derived category, since it is an isomorphism on each cohomology group. We thus obtain a triangle
\begin{displaymath}
E^r_k(F) \to M \to N' \to
\end{displaymath}
where $N'=E^r_k(N)$. Now, $E^r_k(F)$ has finite regularity (Proposition~\ref{prop:sumsupfin}), and so $\reg_{n+1}(N')$ is finite. Similarly, $E^r_k(F)$ is $\tau_r$-finite, and so $N'$ is as well. Consider the exact triangle
\begin{displaymath}
\rR \Gamma_{<r}(E^r_k(F)) \to \rR \Gamma_{<r}(M) \to \rR \Gamma_{<r}(N') \to
\end{displaymath}
The right term is $\tau_{r-1}$-finite by Proposition~\ref{prop:taubd2}, and the left term is $\tau_{r-1}$-finite by Proposition~\ref{prop:sumsupfin}. Thus the middle term is $\tau_{r-1}$-finite, which completes the proof.
\end{proof}

\subsection{Bounds on regularity of derived saturation}

Fix an integer $0 \le r \le d$. Let $Y=\Gr_r(E)$, let $\cQ$ be the tautological bundle, and let $\pi \colon Y \to \Spec(\bC)$ be the structure map. Let $B=\bA(\cQ)$ and let $S' \colon \Mod_B^{\gen} \to \Mod_B$ be the section functor. For $M \in \rD^b(B)$, we define $\reg(M;B)$ to be the minimal integer $\rho$ such that $\rH^{-i}(M \otimes^{\rL}_B \cO_Y)$ is supported in degrees $\le i+\rho$ for all $i$.

\begin{lemma} \label{lem:regsat1}
Let $M \in \rD^b(B)$. Suppose that $\reg(M;B)$ is finite. Then $\reg(\rR \pi_*(M); A)$ is also finite.
\end{lemma}

\begin{proof}
This more or less follows from the methods in \cite[\S 5]{weyman}. To give details, we use the results on Koszul duality in \cite{symu1}, one of which encodes this method into the present setting. First, note that there is a surjection $E \otimes \cO_Y \to \cQ$ of vector bundles, and thus a surjection of tca's $\pi^*(A) \to B$. Thus $B$ is a finitely generated $\pi^*(A)=\bA(\pi^*(E))$-module, and as such, has finite regularity over $\pi^*(A)$. By an argument similar to the one in Proposition~\ref{prop:cmpreg}, we see that $\reg(M;\pi^*(A))$ is finite. Now, let $\sK_{\pi^*(E)}$ and $\sK_E$ be the Koszul duality functors on $\pi^*(A)$- and $A$-modules, respectively, as defined in \cite[\S 7.1.1]{symu1}. Since $\reg(M; \pi^*(A))$ is finite, it follows from \cite[Proposition~7.1]{symu1} that $\sK_{\pi^*(E)}(M)$ is a bounded complex, i.e., it has only finitely many non-zero cohomology groups. Since $Y$ is finite dimensional, it follows that $\rR \pi_*(\sK_{\pi^*(E)})(M)$ is bounded. By \cite[Proposition~7.5]{symu1}, Koszul duality commutes with pushforward, and so $\sK_E(\rR \pi_*(M))$ is bounded. But this exactly means that $\rR \pi_*(M)$ has finite regularity over $A$, by \cite[Proposition~7.1]{symu1} again.
\end{proof}

\begin{lemma} \label{lem:regsat2}
Let $M \in \rD^b(B)$ and let $\cF$ be an $\cO_Y$-module. Suppose that $\reg(M;B)$ is finite. Then $\reg(M \otimes^{\rL}_{\cO_Y} \cF; B)$ is finite.
\end{lemma}

\begin{proof}
First suppose $\cF$ is $\cO_Y$-flat. Then we have
\begin{displaymath}
\rH^{-i}((M \otimes^{\rL}_{\cO_Y} \cF) \otimes^{\rL}_B \cO_Y)=\cF \otimes_{\cO_Y} \rH^{-i}(M \otimes^{\rL}_B \cO_Y),
\end{displaymath}
and so the result follows. If $\cF$ is not flat, pick a short exact sequence
\begin{displaymath}
0 \to \cF_1 \to \cF_2 \to \cF \to 0
\end{displaymath}
where $\cF_2$ is $\cO_Y$-flat. Then $\cF_1$ has strictly smaller Tor-dimension than $\cF$, and so (by induction on Tor-dimension) we can assume that the lemma holds for $\cF_1$. We have a triangle
\begin{displaymath}
M \otimes^{\rL}_{\cO_Y} \cF_1 \to M \otimes^{\rL}_{\cO_Y} \cF_2 \to M \otimes^{\rL}_{\cO_Y} \cF \to
\end{displaymath}
Since the left two terms have finite regularity, so does the third.
\end{proof}

\begin{proposition} \label{prop:regsat}
Let $M \in \rD^b(A)$. Suppose that $\tau_r(M)$ is finite. Then $\reg(\rR \Sigma_{\ge r}(M))$ is finite.
\end{proposition}

\begin{proof} 
By d\'evissage, we can reduce to the case where $M$ is a module. Let $M'=T_{\ge r}(M)$ so that $\rR \Sigma_{\ge r}(M)=\rR S_{\ge r}(M')$. Since $\sigma_r(M') \le \tau_r(M)$, we have $\gr^r_{\lambda}(M')=0$ for $\vert \lambda \vert \ge \tau_r(M)$. Thus, by a second d\'evissage, we can reduce to the case where $M'=\gr^r_{\lambda}(M')$ for some $\lambda$. Recall that $\Mod_{A,r}$ is equivalent to $\Mod_B^{\gen}$. Let $\cF \otimes \bS_{\lambda}(\cK)$ be the object of $\Mod_B^{\gen}$ corresponding to $M'$ (the object has this form by Proposition~\ref{prop:gr}). We have $\rR \Sigma_{\ge r}(M)=\rR \pi_* \rR S'(\cF \otimes_{\cO_Y} \bS_{\lambda}(\cK))$ by \cite[Proposition~6.8]{symu1}. Since $\bS_{\lambda}(\cK)$ is $\cO_Y$-flat, this tensor product is the same as the derived tensor product.  Let $N=\rR S'(\bS_{\lambda}(\cK))$. Then $\reg(N;B)$ is finite since $N \in \rD^b_{\fgen}(B)$. Since $\rR S'$ is $\cO_Y$-linear \cite[Corollary~5.16]{symu1}, we have
\begin{displaymath}
\rR S'(\cF \otimes_{\cO_Y} \bS_{\lambda}(\cK)) = \cF \otimes^{\rL}_{\cO_Y} N.
\end{displaymath}
By Lemma~\ref{lem:regsat2}, $\reg(\cF \otimes^{\rL}_{\cO_Y} N; B)$ is finite. Thus, by Lemma~\ref{lem:regsat1}, $\reg(\rR \pi_*(\cF \otimes^{\rL}_{\cO_Y} N); A)$ is finite, which completes the proof.
\end{proof}

\subsection{The main theorem}

The dimension of $\Gr_r(E)$ is $r(d-r)$. For $d$ fixed, this is maximized at $r=\tfrac{1}{2} d$, with maximum value $\tfrac{1}{4} d^2$. It follows that the maximum dimension of $\Gr_r(E)$ is $\lfloor \tfrac{1}{4} d^2 \rfloor$.

\begin{theorem} \label{thm:main2}
Let $M \in \rD^{\ge 0, b}(A)$. Let $n=\lfloor \tfrac{1}{4} d^2 \rfloor$. Suppose that $\reg_{n+1}(M)$ is finite. Then $\reg(\rR \Gamma_{\le r}(M))$ and $\reg(\rR \Gamma_{\ge r}(M))$ and $\tau_r(\rR \Gamma_{\le r}(M))$ are finite for all $r$.
\end{theorem}

\begin{proof}
Consider the following two statements, for $M \in \rD^{\ge 0, b}(A)$:
\begin{itemize}
\item[$\cA(r)$:] If $\reg_{n+1}(M)$ is finite then $\reg(\rR \Sigma_{\ge r}(M))$ is finite.
\item[$\cB(r)$:] If $\reg_{n+1}(M)$ is finite then $\tau_r(\rR \Gamma_{\le r}(M))$ is finite.
\end{itemize}
We prove these statements by descending induction on $r$. For $r>d$, we have $\rR \Sigma_{\ge r}(M)=0$, and so $\cA(r)$ obviously holds. For $r=d$, we have $\rR \Gamma_{\le r}(M)=M$, and $\tau_d(M)$ is finite since $\reg_0(M)$ is finite (Proposition~\ref{prop:taud}). Thus $\cB(d)$ holds.

Suppose now that both $\cA(r+1)$ and $\cB(r)$ hold. Since $\tau_{r}(\rR \Gamma_{\le r}(M))$ is finite by $\cB(r)$, Proposition~\ref{prop:regsat} gives that $\rR \Pi_r(M)=\rR \Sigma_{\ge r}(\rR \Gamma_{\le r}(M))$ has finite regularity. Consider the triangle
\begin{displaymath}
\rR \Pi_r(M) \to \rR \Sigma_{\ge r}(M) \to \rR \Sigma_{>r}(M) \to
\end{displaymath}
We have just explained that the left term has finite regularity. The right term has finite regularity by $\cA(r+1)$. It follows that the middle term has finite regularity, which proves $\cA(r)$.

Maintain the assumption that $\cA(r+1)$ and $\cB(r)$ hold. Consider the triangle
\begin{displaymath}
\rR \Gamma_{\le r}(M) \to M \to \rR \Sigma_{>r}(M) \to
\end{displaymath}
Since the right term has finite regularity by $\cA(r+1)$, we conclude that the left term has finite $\reg_{n+1}$. The left term also has finite $\tau_r$ by $\cB(r)$. It now follows from Proposition~\ref{prop:taubd} that $\rR \Gamma_{<r}(M)=\rR \Gamma_{<r}(\rR \Gamma_{\le r}(M))$ has finite $\tau_{r-1}$. Thus $\cB(r-1)$ holds.

We have thus shown that $\cA(r+1)$ and $\cB(r)$ implies $\cA(r)$ and $\cB(r-1)$. We conclude that $\cA(r)$ and $\cB(r)$ hold for all $r$. From $\cA(0)$, we see that $M$ itself has finite regularity. Combining this with $\cA(r)$ and the usual exact triangle, we see that $\rR \Gamma_{\le r}(M)$ also has finite regularity for all $r$.
\end{proof}

\begin{corollary}
Let $M$ be as in Theorem~\ref{thm:main2}. Then $\reg(M)$ is finite.
\end{corollary}

\begin{proof}
We have $M=\Gamma_{\le d}(M)=\Sigma_{\ge 0}(M)$.
\end{proof}

\begin{corollary}[Theorem~\ref{thm:deg}]
Let $M$ be a finitely generated $A$-module and let $c=[M]$ be the class of $M$ in $\rK(A)$. Then one can bound $\deg(c)$ by a function of $\reg_{n+1}(M)$.
\end{corollary}

\begin{proof}
Consider the identification $\rK(A) = \bigoplus_{r=0}^d \Lambda \otimes \rK(\Gr_r(E))$. The projection $\rK(A) \to \Lambda \otimes \rK(\Gr_r(E))$ is computed by applying the functor $\rR \Pi_r = \rR \Sigma_{\ge r} \rR \Gamma_{\le r}$. By Theorem~\ref{thm:main2}, $\tau_r(\rR \Gamma_{\le r}(M))$ is bounded by a function of $\reg_{n+1}(M)$. The K-class of $\rR \Gamma_{\le r}(M)$ is the same as its associated graded with respect to any finite filtration. We use the filtration $\gr^r_\mu$, which is finite since we only need to take $\mu$ with $|\mu|<\tau_r(\rR \Gamma_{\le r}(M))$. By Corollary~\ref{cor:sh-K}, $\gr^r_\mu$ corresponds to $\sh^r_\mu(M) \otimes \bS_\mu(\cK)$ under the equivalence $\Psi \colon \Mod_{A,r}[\fa_r] \to \Mod_B^\gen$, and the degree of this element only depends on $\bS_\mu(\cK)$. Since only finitely many possibilities occur for $\mu$, we get that $\deg(c)$ is bounded by a function of $\reg_{n+1}(M)$.
\end{proof}

\section{Additional results} \label{s:add}

\subsection{General coefficients}

Let $X$ be a complex scheme and $\cE$ a locally free sheaf of constant rank $d$ on $X$. Let $M$ be a $\bA(\cE)$-module.

\begin{proposition} \label{prop:reduction}
Theorem~\ref{thm:main} hold for $\bA(\cE)$-modules.
\end{proposition}

\begin{proof}
First we handle the case that $X$ is affine, so that $X = \Spec(R)$ for a $\bC$-algebra $R$, and $\cE \cong R^d$ is free. Let $M$ be a $A_R = \bA(R^d)$-module. In that case, we prove that 
\[
\max\deg \Tor_n^{A_R}(M,R) = \max\deg \Tor_n^{A_\bC}(M,\bC),
\]
where $\max\deg$ denotes the maximum degree in which a graded module is nonzero, by induction on $n$.

When $n=0$, we are measuring the maximum degree needed for a minimal generating set of $M$ with coefficients either in $R$ or $\bC$. Clearly, we have 
\[
\max\deg \Tor_0^{A_R}(M,R) \le \max\deg \Tor_0^{A_\bC}(M,\bC)
\]
since if $M$ can be generated in degree $\le D$ with $\bC$-coefficients, the same is true with $R$-coefficients. Conversely, suppose $M$ can be generated in degree $\le D$ with $R$-coefficients. We will think of $M$ as an $\FI_d$-module. This means that any element in degree $>D$ can be written as a linear combination $\sum_i r_i \phi_i (x_i)$ where $r_i \in R$, $\phi_i$ is an $\FI_d$-morphism, and $x_i \in M_{\le D}$. Since the $\FI_d$-morphisms are all $R$-linear, this is the same as $\sum_i \phi_i(r_ix_i)$ and $r_i x_i$ has the same degree as $x_i$, so $M$ can also be generated in degree $\le D$ with $\bC$-coefficients.

For the case of $n>0$, pick a short exact sequence of $A_R$-modules
\[
0 \to N \to F \to M \to 0
\]
where $F$ is a free $A_R$-module and $F \to M$ is a minimal surjection. Then $\Tor_n^{A_R}(M,R) = \Tor_{n-1}^{A_R}(N,R)$. By induction on $n$, we have $\max \deg \Tor_{n-1}^{A_R}(N,R) = \max \deg \Tor_{n-1}^{A_\bC}(N,\bC)$. The notion of minimal morphism is insensitive to the coefficients (it amounts to knowing the map has positive degree), so we also conclude that $\Tor_n^{A_\bC}(M,\bC) = \Tor_{n-1}^{A_\bC}(N,\bC)$, which finishes the proof in the case that $X$ is affine and $\cE$ is free.

Finally, for the general case, we can pick an open affine covering $\{U_i\}_{i \in I}$ of $X$ such that $\cE$ is free over each $U_i$. Since the calculation of Tor is local, we have 
\[
\max \deg \Tor_n^{\bA(\cE)}(M, \cO_X) = \max_{i \in I} \max \deg \Tor_n^{\bA(\cE|_{U_i})}(M|_{U_i}, \cO_{U_i}),
\]
so that it suffices to know that the theorem holds for each open affine $U_i$.
\end{proof}

\subsection{Proof of Theorem~\ref{thm:stable}}

Recall the statement: for any $n$, $\reg(M_n)$ is bounded by a function of $\reg(M_k)$ where $k =  \ell(M) + \lfloor \tfrac{1}{4} d^2 \rfloor+d+1$. To see this, note that $\Tor_i^A(M,\bC)$ is a subquotient of $\bigwedge^i(\bC^d \otimes \bV) \otimes A \otimes M$, and hence
\begin{displaymath}
\ell(\Tor_i^A(M,\bC)) \le i + d + \ell(M).
\end{displaymath}
Thus for $i \le N=\lfloor \tfrac{1}{4} d^2 \rfloor+1$ we have $\ell(\Tor_i^A(M,\bC)) \le k$, and so one can evaluate $\Tor_i^A(M,\bC)$ on $\bC^k$ without losing information. It follows that $\reg_N(M) = \reg_N(M_k) \le \reg(M_k)$. By Theorem~\ref{thm:main}, $\reg(M)$ is bounded by a function of $\reg_N(M)$, and thus a function of $\reg(M_k)$. Since $\reg(M_n) \le \reg(M)$ for all $n$, the result follows.

\end{document}